\documentclass[10pt]{amsart}            % for proof reading

\usepackage{amssymb,amsmath,amsthm}
\usepackage{hyperref}
\usepackage{color}
\usepackage{mathrsfs}

\makeatletter
\def\blfootnote{\gdef\@thefnmark{}\@footnotetext}
\makeatother

\renewcommand\mathcal{\mathscr}

\usepackage{ulem}
\renewcommand{\emph}{\normalem}

\theoremstyle{plain}
\newtheorem{theorem}{Theorem}[section]
\newtheorem*{theorem*}{Theorem}
\newtheorem{lemma}[theorem]{Lemma}
\newtheorem*{lemma*}{Lemma}

\newtheorem{proposition}[theorem]{Proposition}

\theoremstyle{remark}

\newtheorem*{remark*}{Remark}

\theoremstyle{definition}
\newtheorem{definition}[theorem]{Definition}
\newtheorem*{definition*}{Definition}

\newtheorem{Basic assumptions}[theorem]{Basic assumptions}

\numberwithin{equation}{section}

\usepackage{graphicx}        % standard LaTeX graphics tool
\usepackage[bottom]{footmisc}% places footnotes at page bottom

\usepackage{newtxtext}      % 
\usepackage{newtxmath}      % selects Times Roman as basic font

\newcommand{\NN}{\mathbb{N}}

\newcommand{\CZ}{Calder\'on--Zygmund }

\newcommand{\jl}{j_{\ell}}
\newcommand{\di}{\,{\rm{d}}}
\newcommand{\opnorm}{\||}

\begin{document}

\title[Hardy spaces on weighted homogeneous trees]{Hardy spaces on weighted homogeneous trees}  

%\subjclass[2000]{} 


\thanks{The second and third authors are members of the Gruppo Nazionale per l'Analisi Matematica, 
la Probabilit\`a e le loro Applicazioni (GNAMPA) of the Istituto 
Nazionale di Alta Matematica (INdAM).
}

\author[L. Arditti, A. Tabacco and M. Vallarino]
{Laura Arditti, Anita Tabacco amnd Maria Vallarino}

\address{Laura Arditti, Anita Tabacco, Maria Vallarino:
Dipartimento di Scienze Matematiche 
"Giuseppe Luigi Lagrange"
\\ Politecnico di Torino\\
corso Duca degli Abruzzi 24\\ 10129 Torino\\  Italy
\hfill\break
laura.arditti@polito.it, anita.tabacco@polito.it, 
maria.vallarino@polito.it}

\begin{abstract}
We consider an infinite homogeneous tree $\mathcal V$ endowed with the usual metric $d$ defined on graphs and a weighted measure $\mu$. The metric measure space $(\mathcal V,d,\mu)$ is nondoubling and of exponential growth, hence the classical theory of Hardy spaces does not apply in this setting. We construct an atomic Hardy space $H^1(\mu)$ on $(\mathcal V,d,\mu)$ and investigate some of its properties, focusing in particular on real interpolation properties and on boundedness of singular integrals on $H^1(\mu)$.  
\end{abstract}

%VERSIONE DI \today
\maketitle

%\tableofcontents

\section{Introduction}
	Let $\mathcal V$ be an infinite homogeneous tree of order $q+1$ endowed with the usual distance $d$ defined on a graph (see Section \ref{notation} for the precise definitions). 
	Fix a doubly-infinite geodesic $g$ in $\mathcal V$ and define a mapping $N: g \rightarrow \mathbb{Z}$ such that 
	\begin{equation}
		\left| N(x) - N(y) \right|  = d(x,y) \qquad \forall x,y \in g \,.
	\end{equation}		
%This corresponds to the choice of an origin $o \in g$ (the only vertex for which $N(o)=0$) and an orientation for $g$. %; in this way we obtain a numeration of the vertices in $g$.
	We define the level function $\ell: \mathcal{V} \rightarrow \mathbb{Z}$ as
	\[
	\ell(x) = N(x') - d(x,x')\,,
	\]
	where $x'$ is the only vertex in $g$ such that $d(x,x') = \min\lbrace d(x,z): z\in g \rbrace$. Let $\mu$ be the measure on $\mathcal{V}$ defined by
	\begin{equation}
		\int_{\mathcal{V}} f \, d\mu = \sum_{x \in \mathcal{V}} f(x) q^{\ell(x)}
	\end{equation}
for every function $f$	 defined on $\mathcal V$. Then $\mu$ is a weighted counting measure. We shall show in Subsection \ref{SecMeasureSphere} that the space $(\mathcal V,d,\mu)$ is nondoubling and it is of exponential growth. In particular on such space the classical \CZ theory does not hold. 
 
Hebisch and Steger \cite{HS} developed a new \CZ theory which can be applied also to nondoubling metric measure spaces and showed that such a theory can be 
applied to the space $(\mathcal V,d,\mu)$.  In particular they proved that there exists a family of appropriate sets in $\mathcal V$, which are called \CZ sets, which replace the family of balls in the classical \CZ theory. We mention also that some properties of the space $(\mathcal V,d,\mu)$ were investigated in more detail in~\cite{Ar}.  
 
The purpose of this work is to develop a 
theory of Hardy spaces on $(\mathcal V,d,\mu)$, which is a natural 
development of the \CZ theory introduced in \cite{HS}. Following the classical atomic definition of Hardy spaces 
\cite{CW}, for each $p$ in $(1,\infty]$ we define an atomic Hardy space
$H^{1,p}(\mu)$.  Atoms are functions supported in \CZ sets, with 
vanishing integral and satisfying a certain size condition. 
We shall prove that all the spaces $H^{1,p}(\mu)$, $p\in (1, \infty]$, coincide and we simply denote by $H^1(\mu)$ this atomic Hardy space.  

We then find the real interpolation spaces between $H^1(\mu)$ and $L^q(\mu)$, $q\in (1,\infty]$. The interpolation results which 
we prove are the analogue of the classical interpolation results (see \cite{Ha,J,P,RS}), but the proofs are different. 
Indeed, in the classical setting the maximal characterization 
of the Hardy space is used to obtain the interpolation results, 
while the Hardy space $H^1(\mu)$ introduced in this paper has an atomic definition.

Further, we show that a singular integral operator whose kernel satisfies an 
integral H\"ormander condition, extends to a bounded
operator from $H^1(\mu)$ to $L^1(\mu)$.  
As a consequence of this result, we show that spectral multipliers
of a distinguished Laplacian $\mathcal L$ and the first order Riesz transform associated to $\mathcal L$ extend to bounded operators from $H^1(\mu)$ to
$L^1(\mu)$.

It would be also interesting to characterize the dual space of $H^1(\mu)$ and to obtain complex interpolation results involving $H^1(\mu)$, its dual and the $L^q(\mu)$-spaces. This will be the object of further investigations.

\bigskip

All the results described above may be considered as an analogue of the classical theory of Hardy spaces.%  introduced on spaces of homogeneous type in \cite{CW}. 

The classical Hardy space \cite{CW, FS, S} was introduced in $(\mathbb R^n,d,m)$, where $d$ is the 
Euclidean metric and $m$ denotes the Lebesgue measure and more generally on a 
space of homogeneous type, i.e. a metric measure space $(X,d,\mu)$ where the 
doubling condition is satisfied, i.e., there exists
a constant $C$ such that
\begin{equation}\label{doubling}
\mu\bigl(B(x,2r)\bigr)
\leq C\, \mu\bigl(B(x,r)\bigr)
\qquad\forall x \in X\,, \qquad\forall r >0.
\end{equation}

Extensions of the theory of Hardy spaces have been considered 
in the literature on various metric measure spaces which do not satisfy the doubling condition (\ref{doubling}). The literature on this subject is huge and we shall only cite here some 
contributions \cite{ CMM, CM, MOV, V} which are strictly related to our work.

%Carbonaro, Mauceri and Meda \cite{CMM} defined a Hardy space adapted to any metric measure space which satisfies some geometric assumption, namely the local dopubl;ind property, the isoperimetric property and the approximate midpoint property. Their theory does not apply to the space $(\mathcal V,d,\mu)$ since it does not satisfy the isoperimetric property (see Subsection \ref{SubSecIsop}).  

In particular, we mention that Celotto and Meda \cite{CM} studied various Hardy spaces on a homogeneous tree $\mathcal V$ endowed with the metric $d$ and the counting measure, which is not the measure $\mu$ that we consider here. Their theory is useful to study the boundedness of singular integral operators related to the standard Laplacian defined on trees which is self-adjoint with respect to the counting measure and not to the measure $\mu$. The theory we develop here instead is useful to study singular integral operators related to a distinguished Laplacian self-adjoint on $L^2(\mu)$ (see Subsection \ref{SubSecSingInt}). 

We mention that in \cite{MOV, V} the authors used the \CZ theory introduced by Hebisch and Steger in \cite{HS} to construct Hardy spaces on some solvable Lie groups of exponential growth and studied their properties. Our work can be thought as a counterpart in a discrete setting of the results in \cite{V}, and some of our proofs are strongly inspired by it. 

 \bigskip

 Positive constants are denoted by $C$; these may differ from one line to another, and may depend on any quantifiers written, implicitly or explicitly, before the relevant formula.

	\section{Weighted homogeneous trees} \label{notation}
	In this section we introduce the infinite homogeneous tree and we define a distance $d$ and and a measure $\mu$ on it. We show that the corresponding metric measure space $(\mathcal{V}, d, \mu)$ does not satisfy the doubling property. We then introduce a family of sets, called trapezoids, which will be fundamental in the construction of Hardy spaces.
	\begin{definition}
		An infinite homogeneous tree of order $q+1$ is a graph $T=(\mathcal{V},\mathcal{E})$, where $\mathcal V$ denotes the set of vertices and $\mathcal E$ denotes the set of edges, with the following properties:
		\begin{enumerate}
			\item[(i)] $T$ is connected and acyclic;
			\item[(ii)] each vertex has exactly $q+1$ neighbours.
		\end{enumerate}
	\end{definition}	 
	
%	\begin{figure} [h]
%		\centering
%		\includegraphics[width=0.4\linewidth]{pics/tree}
%		\caption{A representation of the infinite homogeneous tree of order $3$.}
%		\label{fig:tree}
%	\end{figure}
	
	On $\mathcal{V}$  we can define the distance $d(x,y)$ between two vertices $x$ and $y$ as the length of the shortest path between $x$ and $y$.
	We also fix a doubly-infinite geodesic $g$ in $T$, that is a connected subset $g \subset \mathcal{V}$ such that
	\begin{enumerate}
		\item[(i)] for each element $v \in g$ there are exactly two neighbours of $v$ in $g$;
		\item[(ii)] for every couple $(u,v)$ of elements in $g$, the shortest path joining $u$ and $v$ is contained in $g$.
	\end{enumerate}
	We define a mapping $N: g \rightarrow \mathbb{Z}$ such that 
	\begin{equation}
		\left| N(x) - N(y) \right|  = d(x,y) \qquad \forall x,y \in g \,.
	\end{equation}		
This corresponds to the choice of an origin $o \in g$ (the only vertex for which $N(o)=0$) and an orientation for $g$; in this way we obtain a numeration of the vertices in $g$.
	We define the level function $\ell: \mathcal{V} \rightarrow \mathbb{Z}$ as
	\[
	\ell(x) = N(x') - d(x,x')
	\]
	where $x'$ is the only vertex in $g$ such that $d(x,x') = \min\lbrace d(x,z): z\in g \rbrace$.
	For $x,y \in \mathcal{V}$ we say that $y$ \textit{lies above} $x$ if
	\[
	\ell(x) = \ell(y) - d(x,y).
	\]
	In this case we also say that $x$ \textit{lies below} $y$.
	
	\begin{figure}[!tbp]
		\label{fig:pic1}
		\begin{center}		\includegraphics[width=0.9\linewidth]{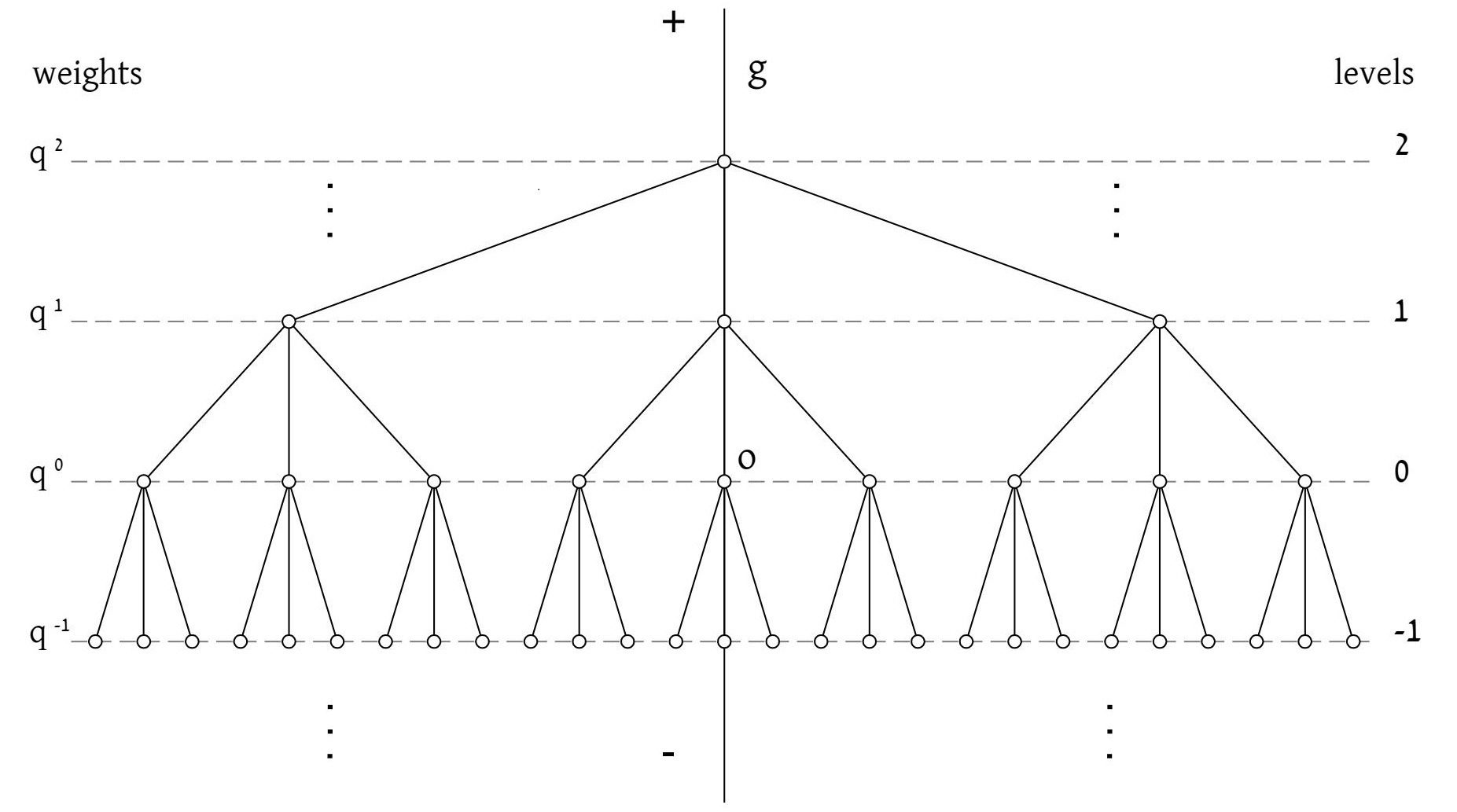}
			\caption{Representation of the measure $\mu$ ($q=3$)}
		\end{center}
	\end{figure}

	Let $\mu$ be the measure on $\mathcal{V}$ such that for each function $f: \mathcal{V}\rightarrow\mathbb C$
	\begin{equation}
		\int_{\mathcal{V}} f \, d\mu = \sum_{x \in \mathcal{V}} f(x) q^{\ell(x)}.
	\end{equation}
	Then $\mu$ is a weighted counting measure such that the weight of a vertex depends only on its level and the weight associated to a certain level is given by $q$ times the weight of the level immediately underneath (see Figure 1).%\ref{fig:pic1}).
	
	\subsection{Doubling and local doubling properties} \label{SecMeasureSphere}
	Observe that the space $(\mathcal{V},d,\mu)$ exhibits exponential volume growth. Indeed given $x_0\in\mathcal V$ and $r\geq 1$ consider the sphere $	S_r(x_0) = \left\lbrace x \in \mathcal{V} : d(x,x_0) = r \right\rbrace $	and the closed ball $B_r(x_0) = \left\lbrace x \in \mathcal{V} : d(x,x_0) \leq r \right\rbrace.$ A direct computation shows that for $r\geq 1$  their measures are given by:	
	\[
		\mu(S_r(x_0)) = q^{\ell(x_0)+r-1}(1+q) 
		\qquad \textrm{and} \qquad
		\mu(B_r(x_0)) =q^{\ell(x_0)} \, \frac{q^{r+1}+q^r-2}{q-1}.
	\]
We notice that they depend on the level of the center $x_0$ and grow exponentially with respect to the radius $r$. As a consequence we can prove the following.
	
	\begin{proposition}
		The space $(\mathcal V,d,\mu)$ is not doubling but it is locally doubling.
	\end{proposition}
	
	\begin{proof}
		Fix $x_0 \in \mathcal{V}$ and notice that
		\begin{align*}
		\lim_{r\rightarrow\infty} \frac{\mu(B_{2r}(x_0))}{\mu(B_r(x_0))} = \lim_{r\rightarrow\infty} \frac{q^{\ell(x_0)} \,\frac{q^{2r+1}+q^{2r}-2}{q-1}}{q^{\ell(x_0)} \,\frac{q^{r+1}+q^r-2}{q-1}}
		= \lim_{r\rightarrow\infty} q^r = +\infty.
		\end{align*}
		Thus the doubling property (\ref{doubling}) fails.
		
		Instead, we show that $(\mathcal V,d,\mu)$ is locally doubling. Indeed, fix $x_0 \in \mathcal{V}$ and $R>0$ and consider $r \leq R$; one has
		\begin{align*}
		\mu(B_{2r}(x_0)) &= q^{\ell(x_0)} \, \frac{q^{2r+1}+q^{2r}-2}{q-1}
		\leq q^{\ell(x_0)} \frac{\frac{q^{2R+1}+q^{2R}-2}{q-1}}{\frac{q^{R+1}+q^{R}-2}{q-1}} \frac{q^{r+1}+q^{r}-2}{q-1} \\
		&= C_R\mu(B_{r}(x_0)) 
		\end{align*}
		with $C_R = \frac{q^{2R+1}+q^{2R}-2}{q^{R+1}+q^{R}-2} >0 $ independent of $x_0$ and $r$.
	\end{proof}

	\subsection{Admissible trapezoids and Calder\'on--Zygmund sets}
	In this subsection we introduce the notion of trapezoid and recall the definition and the main properties of the admissible trapezoids introduced in \cite{HS}.
	
	\begin{definition}
		We call trapezoid a set of vertices $S \subset \mathcal V$ for which there exist $x_{S} \in \mathcal V$ and $a,b \in \mathbb R_+$ such that
		\begin{equation}
		S= \left\lbrace x \in \mathcal{V} : x \text{ lies below } x_S \,, a \leq \ell(x_S)-\ell(x) < b \right\rbrace.
		\end{equation}
	\end{definition}
	In the following we will refer to $x_S$ as the root node of the trapezoid.
	Among all trapezoids we are mostly interested in those where $a$ and $b$ are related by particular conditions, as specified in the following definitions.
	
	\begin{definition} \label{DefAdmissibleTrapezoid}
		A trapezoid $R \subset \mathcal{V}$ is an admissible trapezoid if and only if one of the following occurs:
		\begin{enumerate}
			\item[(i)] $R=\left\lbrace x_R \right\rbrace$ with $x_R \in \mathcal{V}$, that is $R$ consists of a single vertex ;
			\item[(ii)] $\exists x_R \in \mathcal{V}, \, \exists h \in \mathbb{N}^+$ such that
			\[
			R= \left\lbrace x \in \mathcal{V} : x \text{ lies below } x_R \,, h \leq \ell(x_R)-\ell(x) < 2h \right\rbrace.
			\]
		\end{enumerate}
	\end{definition}
	We set $h(R)=1$ in the first case and $h(R)=h$ in the second case. In both cases $h(R)$ can be interpreted as the height of the admissible trapezoid, which coincides with the number of levels spanned by $R$ (see Figure 2).
	
	\begin{definition}
		We call width of the admissible trapezoid $R$ the quantity
		\[
		w(R) = q^{\ell(x_R)}.
		\]
	\end{definition}
	
	We have that:
	\begin{equation} \label{PropMeasureOfTrapezoid}
		\mu(R) = h(R)q^{\ell(x_R)} = h(R)w(R).
	\end{equation}
	
	\begin{figure}[!btp]
		\begin{center}		\includegraphics[width=1.1\linewidth]{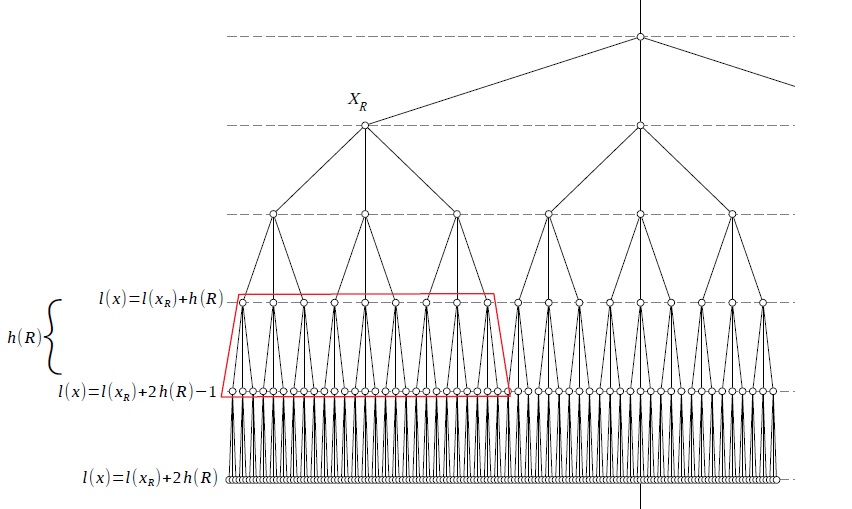} 
			\label{fig:pic2}
			\caption{Representation of an admissible trapezoid with $h(R) = 2$ ($q=3$)}
		\end{center}
	\end{figure}

We now introduce the family of Calder\'on--Zygmund sets. They are trapezoids, even if not of admissible type; they consist of suitable enlargements of admissible trapezoids, constructed according to the following definition.

\begin{definition}
	Given an admissible trapezoid $R$, the envelope of $R$ is the set
	\begin{equation}
	\tilde{R}= \left\lbrace x \in \mathcal{V} : x \text{ lies below } x_R \,, \frac{h}{2} \leq \ell(x_R)-\ell(x) < 4h \right\rbrace
	\end{equation}
	and we set $h(\tilde{R})=h(R)$. The envelope of an admissible trapezoid is also called a Calder\'on--Zygmund set.
\end{definition}
	
%	
%	
%	\begin{definition} A set $\tilde{Q} \subset \mathcal{V}$ is a Calder\'on--Zygmund set if and only if one of the following occurs:
%		\begin{enumerate}
%			\item[(i)] $\tilde{Q}=\left\lbrace x_{\tilde{Q}}
%			\right\rbrace$ with $x_{\tilde{Q}} \in \mathcal{V}$, that is $\tilde{Q}$ consists of a single vertex;
%			\item[(ii)] $\exists x_{\tilde{Q}}\in \mathcal{V}, \, \exists h \in \mathbb{N}^+$ such that
%			\[
%			\tilde{Q}= \left\lbrace x \in \mathcal{V} : x \text{ lies below } x_{\tilde{Q}} \,, \frac{h}{2} \leq \ell(x_{\tilde{Q}})-\ell(x) < 4h \right\rbrace.
%			\]
%		\end{enumerate} 
%	\end{definition}
%We set $h(\tilde Q)=1$ in the first case and $h(\tilde Q)=h$ in the second case. 
%	
%	Calder\'on--Zygmund sets are trapezoids, even if not of admissible type. They consist of suitable enlargements of admissible trapezoids, constructed according to the following definition.
%	
%	\begin{definition}
%		Given an admissible trapezoid $R$, we denote by $\tilde{R}$ the Calder\'on--Zygmund set such that:
%		\begin{enumerate}
%			\item[(i)] if $R=\{x_R\}$, then $\tilde{R}=R$;
%			\item[(ii)]  otherwise, $x_{\tilde{R}}=x_R$ and $h(\tilde{R}) = h(R)$.
%		\end{enumerate}
%	We say that $\tilde{R}$ is the envelope of the trapezoid $R$.
%	\end{definition}

	\begin{proposition}\label{mutildeR}
		Let $R$ be an admissible trapezoid. Then:
		\begin{equation}
		\mu(\tilde{R}) \leq 4\mu(R).
		\end{equation}
	\end{proposition}
	
	\begin{proof}
		In the degenerate case one has $R = \lbrace x_R \rbrace = \tilde{R}$ and then $\mu(\tilde{R})=\mu(R)$.
		In the nondegenerate case
		\begin{align*}
		\mu(\tilde{R}) &= \sum_{\ell = \ell(x_R)-4h+1}^{\lfloor \ell(x_R) -\tfrac{h}{2} \rfloor} \sum_{x \in \tilde{R} : \ell(x)=\ell} q^\ell = \sum_{\ell = \ell(x_R)-4h+1}^{\lfloor \ell(x_R) -\tfrac{h}{2} \rfloor} q^\ell q^{\ell(x_R) - \ell} \\
		&\leq q^{\ell(x_R)} \left( \ell(x_R) - \tfrac{h}{2} - \ell(x_R) + 4h \right) \\
		& \leq 4\mu(R) \,,
		\end{align*}
		which concludes the proof.
	\end{proof}
	
	\begin{proposition}\label{inclusion}
		
		Let $R_1$ and $R_2$ be two admissible trapezoids.
		If
		\[
		R_1 \cap R_2 \neq \emptyset \text{   and   } w(R_1)\geq w(R_2) \,,
		\]
		then
		\[
		R_2 \subset \tilde{R_1}.
		\]
	\end{proposition}
	\begin{proof}
	The only nontrivial case is when neither $R_1$ nor $R_2$ is composed of a single vertex.
	Let $x_{R_1}$ and $x_{R_2}$ be the two root nodes of $R_1$ and $R_2$, respectively. Then
	\[
	w(R_1) = q^{\ell(x_{R_1})} \geq q^{\ell(x_{R_2})} = w(R_2) \quad \implies \quad \ell(x_{R_1}) \geq \ell(x_{R_2}).
	\]
	Moreover, since $R_1 \cap R_2 \neq \emptyset$,  $x_{R_2}$ is below $x_{R_1}$ and so is every vertex of $R_2$. In the following we denote $h_1 = h(R_1)$ and $h_2 = h(R_2)$.
	Let $\hat{x} \in R_1 \cap R_2 \neq \emptyset$. Then we obtain the following constraints:
	\[ 
	\begin{cases}
	\ell(x_{R_2}) -2h_2 +1 \leq \ell(\hat{x}) \leq \ell(x_{R_1}) - h_1 \\
	\ell(x_{R_1}) -2h_1 +1 \leq \ell(\hat{x}) \leq \ell(x_{R_2}) - h_2
	\end{cases}
	\implies
	\begin{cases}
	\ell(x_{R_1})-\ell(x_{R_2}) \geq h_1 - 2h_2 + 1\\
	\ell(x_{R_1})-\ell(x_{R_2})  \leq 2h_1 - h_2 -1 \,.
	\end{cases}
	\]
	Let $x \in R_2$; then $x$ lies below $x_{R_1}$. Moreover
	
	\begin{align*}
	\ell(x_{R_1}) - \ell(x) &= \left[ \ell(x_{R_1}) - \ell(x_{R_2}) \right] + \left[ \ell(x_{R_2}) - \ell(x) \right] \\
	& \leq 2\left[ 2h_1 - h_2 - 1 \right] + \left[ 2h_2 - 1 \right] 
	< 4h_1.
	\end{align*}
	\begin{align*}
	\ell(x_{R_1}) - \ell(x) &= \left[ \ell(x_{R_1}) - \ell(x_{R_2}) \right] + \left[ \ell(x_{R_2}) - \ell(x) \right] \\
	& \geq \frac{1}{2} \left[ h_1 - 2h_2 + 1\right] + \left[ h_2 \right]
	> \frac{h_1}{2}.
	\end{align*}
	So $\frac{h_1}{2} \leq \ell(x_{R_1}) - \ell(x) < 4h_1$, $\forall x \in R_2$ and this shows that $R_2 \subset \tilde{R_1}$.
	\end{proof}
	
	\begin{proposition}\label{CZinclusion}
		Given a Calder\'on--Zygmund set $\tilde{R}$, we have that for all $z \in \tilde{R}$
		\[\tilde{R} \subset B(z, 8h(\tilde{R})) \,.\] 
	\end{proposition}
	\begin{proof}
		Fix a point $z\in \tilde{R}$. Every vertex $y \in \tilde{R}$ has distance $d(z,y) \leq 8h(\tilde{R})-2 < 8h(\tilde{R})$. Indeed, starting from $z$ it is possible to reach $y$ passing through at most $\left[ (4h(\tilde{R})-1) \right] 2  = 8h(\tilde{R})-2$ edges, moving from $z$ to the root node of the trapezoid and then from the root node to $y$.
	\end{proof}
	
	\begin{definition}\label{tildeQ*}
		Given a Calder\'on--Zygmund set $\tilde{R}$, we define the set
		\begin{equation}
			\tilde{R}^* = \left\lbrace x \in \mathcal{V} : d(x,\tilde{R}) < h({\tilde{R})}/4 \right\rbrace\,.
		\end{equation}
			\end{definition}
It is easy to see that there exists a positive constant $C$ such that for every Calder\'on--Zygmund set $\tilde{R}$
\begin{equation}\label{mutildeQ*}
\mu(\tilde{R}^*)\leq C\mu(\tilde{R})\,.
\end{equation}	
See \cite[p.75]{Ar} for a proof of this fact. 
	
	\section{The maximal function}\label{SecMaxFunction}
	In this section we define two maximal functions and describe a way to construct a covering of their level sets which will be useful in the sequel.
	
	\begin{definition}
		
		Given $f: \mathcal V \rightarrow \mathbb{C}$, we define the maximal function $M$ as
		\[
		Mf(x) = 
		\sup_{R: x \in R} \frac{1}{\mu(R)} \int_{R} |f| \, d\mu
		\]
		where the supremum is taken over all admissible trapezoids $R$ containing $x$.
	\end{definition}
	
	Consider a function $f \in L^p(\mu)$ and let $\lambda >0$. We are interested in constructing a covering of the level set 
	\[
	\Omega_{\lambda^p} = \lbrace x \in \mathcal{V} : M(|f|^p)(x) > \lambda^p \rbrace \,.
	\]
	Define $S_0$ as the family of all admissible trapezoids $R$ such that 
	\[
	\int_{R} |f|^p \, d\mu \geq \lambda^p \mu(R).
	\]
	Since $S_0$ is countable, we can introduce an ordering in $S_0$.
	All trapezoids in $S_0$ have bounded measure and bounded width, because $\forall R \in S_0$ we have
		\[
		w(R) = \frac{\mu(R)}{h(R)} \leq \mu(R) \leq \frac{1}{\lambda^p} \left\| f \right\|_{L^p}^p\,.
		\]
	So it is possible to choose in $S_0$ a trapezoid $R_0$ of largest width (in case of ties, we choose that trapezoid of largest width which occurs earliest in the ordering).
	Then we proceed inductively:
	\begin{enumerate}
		\item[(i)] $S_{i+1}$ is the family of all admissible trapezoids $R \in S_i$ disjoint from $R_0,\dots,R_i$;
		\item[(ii)] $R_{i+1}$ is the trapezoid of largest width in $S_{i+1}$ which occurs earliest in the ordering.
	\end{enumerate}
		
		Let $R \in S_0$. Then by construction $R$ intersects some $R_i$ with $w(R_i) \geq w(R)$.
		
		Indeed, there exists a number $j \in \lbrace 0,1,2,\dots \rbrace$ such that $R \in S_j$ and $R \notin S_{j+1}$, i.e. in the previous construction there exists a step $j$ in which one of the following occurs:
		\begin{enumerate}
			\item either $R$ is the trapezoid of largest width that occurs earliest in the ordering, and then $R$ is selected and $R_j = R$, so that $R \cap R_i \neq \emptyset$ for $i=j$;
			\item or $R$ is not the trapezoid of largest width that occurs earliest in the ordering and it intersects $R_j$. Then $R$ is not in $S_i \, \, \forall i \geq j+1$ and $R \cap R_i \neq \emptyset$ for $i=j$.
		\end{enumerate}
		To ensure that there is some $j$ with the stated property it is sufficient to avoid that $S_0$ can contain an infinite number of trapezoids with the same width that do not intersect each other. This possibility is excluded observing that:
		\begin{align*}
		\sum_{i} \mu(R_i) &\leq \frac{1}{\lambda^p} \sum_{i}\int_{R_i} |f|^p \, d\mu \leq \frac{1}{\lambda^p}\int_{\mathcal{V}} |f|^p \, d\mu = \frac{1}{\lambda^p} \left\| f \right\|_{L^p} < \infty \,,
		\end{align*}
		while if there was among the $R_i$'s an infinite number of trapezoids with constant width $w$ we would have
		\[
		\sum_{i} \mu(R_i) \geq \sum_{n=1}^{\infty}w = \infty.
		\]
		In conclusion, 
		\[
		\forall R \in S_0, \quad \exists i : R \cap R_i \neq \emptyset \text{ and } w(R_i) \geq w(R).
		\]
		By Proposition \ref{inclusion}, this implies $R \subset \tilde{R_i}$.
		We set $E := \bigcup_i \tilde{R_i}$. We have that $\Omega_{\lambda^p} \subset E $ and
		\begin{equation}
		\mu(E) \leq \sum_{i}\mu(\tilde{R_i}) \leq 4 \sum_{i}\mu(R_i) \leq \frac{4 \left\| f \right\|_{L^p}}{\lambda^p}.
		\end{equation}

	\begin{definition}
		Given $f: \mathcal{V} \rightarrow \mathbb{C}$ and a Calder\'on--Zygmund set $\tilde{Q}$, we define the maximal function $M_{\tilde{Q}}$ as follows
		\[
		M_{\tilde{Q}}(f)(x) = 
		\sup_{R \subset \tilde{Q} \, : \, x \in R}\mu(R)^{-1} \int_{R} |f| \, d\mu \qquad \forall x \in \tilde{Q} \,,
		\]
		where the supremum is taken over all admissible trapezoids $R$ containing $x$ and contained in $\tilde{Q}$. When $x \notin \tilde{Q}$ we set $	M_{\tilde{Q}}(f)(x) = 0$.
	\end{definition}
	
	Consider a function $f \in L^p(\mu)$ with support contained in a Calder\'on--Zygmund set $\tilde{Q}$, and let $\lambda >0$.
	We define  
	\[
	\Omega_{\tilde{Q},\lambda^p} = \lbrace x \in \mathcal{V} : M_{\tilde{Q}}(|f|^p)(x) > \lambda^p \rbrace \,.
	\]
	Arguing as before, we can show that there exists a family of pairwise disjoint admissible trapezoids $\lbrace R_i \rbrace$ such that $R_i \subset \tilde{Q}$, $\sum_{i} \mu(R_i) \leq \frac{1}{\lambda^p} \left\| f \right\|_{L^p}$ and $\Omega_{\tilde{Q},\lambda^p} \subset \bigcup_i \tilde{R_i} $.

	\section{Hardy spaces}\label{Hardy}
In this section we define atomic Hardy spaces replacing balls with \CZ sets in the classical definition of atoms.

\begin{definition}
A function $a$ is a {\emph{$(1,p)$-atom}}, for $ p\in(1, \infty]$, 
if it satisfies the following properties:
\begin{enumerate}
\item[(i)] \,$a$ is supported in a \CZ set $\tilde R$;
\item[(ii)] \,\,$\|a\|_{L^p}\leq \mu (\tilde R)^{1/p-1}\,;$ 
\item[(iii)] \,\,$\int_{\mathcal V} a\di\mu =0$\,.
\end{enumerate}
\end{definition}
Observe that a $(1,p)$-atom is in $L^1(\mu)$ and it is normalized 
in such a way that its $L^1$-norm does not exceed $1$. 

\begin{definition}
The Hardy space $H^{1,p}(\mu)$ is the space of all functions $h$ in $ L^1(\mu)$ 
such that $h=\sum_j \lambda_j\, a_j$, where $a_j$ are $(1,p)$-atoms and $\lambda _j$ 
are complex numbers such that $\sum _j |\lambda _j|<\infty$. We denote by $\|h\|_{H^{1,p}}$ 
the infimum of $\sum_j|\lambda_j|$ over all decompositions $h=\sum_j\lambda_j\,a_j$, 
where $a_j$ are $(1,p)$-atoms. 
\end{definition}
The space $H^{1,p}(\mu)$ endowed with the norm $\|\cdot\|_{H^{1,p}}$ is a Banach space.

\subsection{Equivalence of spaces $H^{1,p}(\mu)$ for $p\in (1,\infty]$}

It easily follows from the above definitions that
$H^{1,\infty}(\mu)\subseteq H^{1,p}(\mu)$,
whenever $p\in (1, \infty)$. Actually we shall prove that $H^{1,\infty}(\mu)= H^{1,p}(\mu)$,
for every $p\in (1, \infty)$. To show this fact we first prove a preliminary result.

 \begin{proposition}\label{atomo}
Let $a$ be a $(1,p)$-atom, where $p\in (1,\infty)$. Then $a$ is in $ H^{1,\infty}(\mu)$ and there exists a positive constant $C_{p}$, which depends only on $p$, such that
$$\|a\|_{H^{1,\infty}}\leq C_{p}\,.$$
\end{proposition}
\begin{proof}
Let $a$ be a $(1,p)$-atom supported in a \CZ set $\tilde Q$. We define $b:=\mu(\tilde Q)\,a$. %Note that $b\in L^p(\mu)$ and $\|b\|_p\leq \mu(\tilde R)^{1/p}$. 

Let $\alpha$ be a positive number such that $\alpha>   2[24 q (1+4^p)]^{1/(p-1)}$. 
 
We  shall prove that for all $n\in \NN$ there exist functions $a_{j_{\ell}}$, $h_{j_n}$ and admissible sets $\tilde R_{j_{\ell}}$, $j_{\ell}\in \NN ^{\ell}$, $\ell=0,...,n$, such that
\begin{equation}\label{b}
b=\sum_{\ell=0}^{n-1} 4(6q)^{1/p}\alpha^{\ell +1}\sum _{j_{\ell}}\mu(\tilde R_{j_{\ell}})\,\,a_{j_{\ell}}+\sum _{j_n\in\mathbb N^n}f_{j_n}\,,
\end{equation}
where the following properties are satisfied:
\begin{enumerate}
\item[(i)] $a_{\jl}$ is a $(1,\infty)$-atom supported in the \CZ set $\tilde R_{\jl}$;
\item[(ii)] $f_{j_n}$ is supported in $R_{j_n}$ and $\int f_{j_n}\di\mu=0$;
\item[(iii)] $\Big(\frac{1}{\mu(\tilde R_{j_n})}\int_{\tilde R_{j_n}}|f_{j_n}|^p\di\mu\Big)^{1/p}\leq  2^{1-1/p}\,(6q)^{1/p}\,(1+4^p)^{1/p}\alpha^n$;
%\item[(iv)]$\sum_{j_n}\|h_{j_n}\|^p_p\leq 2^{pn}\,\|b\|^p_p$;
\item[(iv)] $|f_{j_n}(x)|\leq |b(x)|+ 4(6q)^{1/p}n\,\alpha ^n \,, \qquad \forall x\in \tilde R_{j_n}$;
\item[(v)] $\sum_{j_n}\mu(\tilde R_{j_n})\leq \,   4^{n+1}\,[  2^{p-1}\,(6q)\,(1+4^p) ]^n\alpha ^{-np}\,\|b\|^p_{L^p}$.
\end{enumerate}
 
We first suppose that the decomposition (\ref{b}) exists and we show that $a\in H^{1,\infty}(\mu)$. Set $F_n=\sum_{j_n}f_{j_n}$. We prove that $F_n\in L^1(\mu)$ and that its $L^1$-norm tends to zero when $n$ tends to $\infty$. Indeed, by H\"older's inequality
\begin{align*}
\|F_n\|_{L^1}&\leq\sum_{j_n\in\mathbb N^n}\|f_{j_n}\|_{L^1}\leq\sum_{j_n\in\mathbb N^n}\mu(\tilde R_{j_n})^{1/{p'}}\,\|f_{j_n}\|_{L^p}\,,
\end{align*}
where $p'$ is the conjugate exponent of $p$. Now by (iii) and (v) we have that
$$
\begin{aligned}
\|F_n\|_{L^1}&\leq\sum_{j_n\in\mathbb N^n}\mu(\tilde R_{j_n})^{1/{p'}}\mu(\tilde R_{j_n})^{1/p}\,  2^{1-1/p}\,(6q)^{1/p}\,(1+4^p)^{1/p}\alpha^n\\
%&=\sum_{j_n}\rho(R_{j_n}) M^{n/p}\,2^n\,\alpha ^n\\
&\leq   4^{n+1}\,[  2^{p-1}\,6q\,(1+4^p) ]^n\alpha ^{-np}\,\|b\|^p_{L^p} \,2^{1-1/p}\,(6q)^{1/p}\,(1+4^p)^{1/p}\alpha^n   \,, \\
\end{aligned}
$$
which tends to zero when $n$ tends to $\infty$, since $\alpha>   2[24 q (1+4^p)]^{1/(p-1)}$. 

This shows that the series $\sum_{\ell=0}^{n-1} 4(6q)^{1/p}\alpha^{\ell +1}\sum _{j_{\ell}}\mu(\tilde R_{j_{\ell}})\,\,a_{j_{\ell}}$ converges to $b$ in $L^1(\mu)$. Moreover by (v) we deduce that
\begin{align*}
\sum_{\ell=0}^{\infty}   4(6q)^{1/p}\alpha^{\ell +1}\sum _{j_{\ell}}\mu(\tilde R_{j_{\ell}}) &\leq \sum_{\ell=0}^{\infty}   4(6q)^{1/p}\alpha^{\ell +1} 4^{\ell+1}\,[  2^{p-1}\,(6q)\,(1+4^p) ]^{\ell}\alpha ^{-\ell p}\,\|b\|^p_{L^p} \\
&\leq C_p\,\|b\|^p_{L^p} \leq C_p\,\mu(\tilde Q)\,,
\end{align*}
where $C_{p}$ depends only on $p$.% and $q$. 

It follows that $b$ is in $H^{1,\,\infty}(\mu)$ and $\|b\|_{H^{1,\infty}}\leq C_{p}\,\mu(\tilde Q)\,.$ Thus $a=\mu(\tilde Q)^{-1}\,b$ is in  $H^{1,\,\infty}(\mu)$ and $\|a\|_{H^{1,\infty}}\leq C_{p}$, as required. 

It remains to prove that the decomposition (\ref{b}) exists. We prove it by induction on $n$. %!TEX encoding = UTF-8 UnicodeWe construct a partition $\mathcal P$ of $S$ in \CZ sets which contains the set $R$ (see \cite[Proof of 5.1]{HS}). 

{\bf{Step~$n=1$.}} Define
$$
\Omega_{\tilde Q,\alpha^p}=\{x\in\mathcal V: M_{\tilde Q}(|b|^p)(x)>\alpha^p\}\,.
$$
If $\Omega=\emptyset$, then 
$$
\mu(\{x_R\})^{-1}|b(x_R)|^p\mu(\{x_R\})\leq \alpha^p \,, \qquad \forall x_R\in \mathcal V\,.
$$
It follows that $\|b\|_{L^{\infty}}\leq \alpha$ and we have
$
b=\alpha\mu(\tilde Q)a_{j_0},
$
so that and \eqref{b} is satisfied with the $(1,\infty)$-atom $a_{j_0}=\alpha^{-1}\mu(\tilde Q)^{-1}b$ and $f_{j_1}=0$ for every $j_1\in\mathbb N$.

If $\Omega\neq \emptyset$, then we construct a family of trapezoids $R_{i}$, $i\in\mathbb N$, and the corresponding Calder\'on--Zygmund sets $\tilde R_{i}$, $i\in\mathbb N$, as in Section \ref{SecMaxFunction}. We then define $U_i=\tilde R_i\setminus (\cup_{j<i}\tilde R_j)$, and $h_i=b\chi_{U_i}$. 
One can show as in \cite[p.43]{HS} 
\begin{equation}\label{mediah_i^p}
\int_{U_i}|h_i|^pd\mu\leq 6q \alpha^p \mu(\tilde R_i)\,.
\end{equation}
See also \cite[p.74]{Ar} for a detailed proof of the previous inequality.  

We now define
$$
f_i=h_i- \mu(R_i)^{-1}\int h_i d\mu\,\chi_{R_i} \,, \qquad g=b-\sum_if_i\,.
$$
Notice that $f_i$ is supported in $\tilde Q\cap \tilde R_i$ and $g$ is supported in $\tilde Q$. The average of $f_i$ vanishes by construction. Moreover, for every $x\in\tilde R_i$ we have that 
$$
\begin{aligned}
|f_i(x)|&\leq |h_i(x)|+\mu(R_i)^{-1}\int |h_i |d\mu\chi_{R_i}\\
&\leq |h_i(x)|+\mu(R_i)^{-1} \Big(\int |h_i |^pd\mu\Big)^{1/p}\mu(\tilde R_i)^{1/{p'}}\\
&\leq |b(x)|+(6q)^{1/p} 4\,\alpha\,, 
\end{aligned}
$$
where we have applied \eqref{mediah_i^p} and Proposition \ref{mutildeR}. It follows that 
$$
\|f_i\|_{L^p}^p\leq 2^{p-1}\|h_i\|_{L^p}^p+6q4^p\alpha^p\mu(\tilde R_i)\leq 2^{p-1}6q \alpha^p(1+4^p)\mu(\tilde R_i)\,.
$$
Moreover,
$$
\sum_i\mu(\tilde R_i)\leq 4\sum_i\mu(R_i)\leq 4 \frac{1}{\alpha^p}\sum_i\int_{R_i}|b(x)|^pd\mu(x)\leq 4 \frac{1}{\alpha^p}\|b\|_{L^p}^p\,.
$$
This implies that 
$$
\sum_i\|f_i\|_{L^p}^p\leq 2^{p-1}\,6q \alpha^p(1+4^p)4 \frac{1}{\alpha^p}\|b\|_{L^p}^p\,.
$$
We now estimate the function $g$. If $x$ is a vertex in the complement of $\tilde Q$, then obviously $g(x)=0$. If $x$ is a vertex in $\tilde Q\cap (\bigcup_i \tilde R_i)^c$, then $g(x)=b(x)$ and $M_{\tilde Q}(|b|^p)(x)\leq \alpha ^p$. Thus $|g(x)|=|b(x)|\leq \alpha$\,. Let us now take $x\in \tilde Q\cap (\bigcup_i \tilde R_i)$ and let $i_0$ be the unique index such that $x\in U_{i_0}$. We distinguish two different cases. If $x\notin \bigcup_i R_i$, then $f_{i_0}(x)=h_{i_0}(x)=b(x)$ and $f_i(x)=0$ for every $i\neq i_0$, so that $g(x)=b(x)-b(x)=0$. If $x\in \bigcup_i R_i$, let $i_1$ be the unique index such that $x\in R_{i_1}$. When $i_1=i_0$ we have that $f_i(x)=0$ for every $i\neq i_0$ and 
$$
g(x)=b(x)-f_{i_0}(x)=b(x)-b(x)+\mu(R_{i_0})^{-1}\int h_{i_0}d\mu\,,
$$
so that $|g(x)|\leq (6q)^{1/p}4\alpha$. When $i_1\neq i_0$ we have that $f_i(x)=0$ for every $i\neq i_0,i_1$ and 
$$
g(x)=b(x)-f_{i_0}(x)-f_{i_1}(x)=b(x)-b(x)+\mu(R_{i_1})^{-1}\int h_{i_1}d\mu\,,
$$
so that $|g(x)|\leq (6q)^{1/p}4\alpha$. 

In conclusion, 
$$
b(x)=(6q)^{1/p}4\alpha \mu(\tilde Q) a(x)+\sum_{j_1\in\mathbb N}f_{j_1}(x)\,,
$$
where $a=(6q)^{-1/p}4^{-1}\alpha^{-1} \mu(\tilde Q)^{-1}g$ is a $(1,\infty)$-atom supported in $\tilde Q$ and all properties (i)-(v) are satisfied.

{\bf{Inductive step}}. Suppose that a decomposition 
$$
b=\sum_{\ell=0}^{n-1} 4(6q)^{1/p}\alpha^{\ell +1}\sum _{j_{\ell}}\mu(\tilde R_{j_{\ell}})\,\,a_{j_{\ell}}+\sum _{j_n\in\mathbb N^n}f_{j_n}\,,
$$
holds, where properties (i)-(v) are satisfied. We shall prove that a similar decomposition of $b$ holds with $n+1$ in place of $n$. To do so, we decompose each function $f_{j_n}$ by following the same construction applied above to $b$ with respect to $\alpha^{n+1}$. We omit the details.   
 \end{proof}
 The following theorem is now an easy consequence of Proposition \ref{atomo}. 
 \begin{theorem}\label{coincidono}
For any $p\in (1,\infty)$, $H^{1,p}(\mu)=H^{1,\infty}(\mu)$ and the norms $\|\cdot\|_{H^{1,p}}$ 
and $\|\cdot\|_{H^{1,\infty}}$ are equivalent. 
\end{theorem}
In the sequel we denote by $H^1(\mu)$ the space $H^{1,\infty}(\mu)$ and we define $\|\cdot\|_{H^1}=\|\cdot\|_{H^{1,\infty}}$.

\subsection{Real interpolation properties of $H^1(\mu)$}

In this subsection we study the real interpolation of $H^1(\mu)$ and the $L^p(\mu)$ spaces. We first recall some notation of the real interpolation of normed spaces, focusing on the $K$-method. For the details see \cite{BL}. 

Given two compatible normed spaces $A_0$ and $A_1$, for any $t>0$ and for any $a\in A_0+A_1$ we define
$$K(t,a;A_0,A_1)=\inf\{ \|a_0\|_{A_0}+t\|a_1\|_{A_1}:~a=a_0+a_1,\,a_i\in A_i \}\,.$$ 
Take $q\in [1, \infty]$ and $\theta\in (0,1)$. The {\emph{real interpolation space}} $\big[A_0,A_1\big]_{\theta,q}$ is defined as the set of the elements $a\in A_0+A_1$ such that
$$\|a\|_{\theta,q}=\begin{cases}
\Big(\int_0^{\infty}\big[t^{-\theta}\,K(t,a;A_0,A_1)\big]^q \frac{\di t}{t} \Big)^{1/q}&{\rm{if~}} 1\leq q<\infty\\
\|t^{-\theta}\,K(t,a;A_0,A_1)\|_{\infty}&{\rm{if~}} q=\infty\,,
\end{cases}
$$
is finite. The space $\big[A_0,A_1\big]_{\theta,q}$ endowed with the norm $\|a\|_{\theta,q}$ is an exact interpolation space of exponent $\theta$. 

We refer the reader to \cite{Jo} for an overview of the real interpolation results which hold in the classical setting. Our aim is to prove the same results in our context. Note that in our case a maximal characterization of $H^1(\mu)$ is not avalaible, so that we cannot follow the classical proofs but we shall only use the atomic definition of $H^1(\mu)$ to prove the results.

\smallskip
We shall first estimate the $K$ functional of $L^{p}$-functions with respect to the couple of spaces $(H^1(\mu),L^{p_1}(\mu))$, $1<p_1\leq \infty$.
\begin{lemma}\label{intpinfty}
	Suppose that $1<p< p_1\leq \infty$ and let $\theta\in (0,1)$ be such that $\frac{1}{p}=1-\theta+\frac{\theta}{p_1}$. Let $f$ be in $L^p(\mu)$. The following hold:
	\begin{enumerate}
		\item[(i)] for every $\lambda>0$ there exists a decomposition $f=g^{\lambda}+b^{\lambda}$ in $L^{p_1}(\mu)+H^1(\mu)$ such that
		\begin{enumerate}
			\item[(i')] $\|g^{\lambda}\|_{L^\infty} \leq C\,\lambda$ and, if $p_1<\infty$, then $\|g^{\lambda}\|_{L^p_1}^{p_1}\leq C\,\lambda^{p_1-p}\,\|f\|_{L^p}^p$;
			\item[(i'')] $\|b^{\lambda}\|_{H^1}\leq C\,\lambda^{1-p}\,\|f\|_{L^p}^p\,;  $
		\end{enumerate}
		\item[(ii)] for any $t>0$, $K(t,f;H^1(\mu),L^{p_1}(\mu))\leq C\,t^{\theta}\,\|f\|_{L^p};$
		\item[(iii)] $f\in  [H^1(\mu),L^{p_1}(\mu)]_{\theta,\infty}$ and $\|f\|_{\theta,\infty}\leq C\,\|f\|_{L^p}.$ 
	\end{enumerate}
\end{lemma}
\begin{proof}
	Let $f$ be in $L^p(\mu)$.
	We first prove (i). Given a positive $\lambda$, let 
	$$\Omega_{\lambda^p}=\{x\in\mathcal V: M(|f|^p)(x)>\lambda^p\}\,.
	$$
	Let $\{R_i\}$ be the collection of trapezoids constructed as in Section \ref{SecMaxFunction}. We now define $U_i=\tilde R_i\setminus( \cup_{j<i}\tilde R_j)$ and $h_i=f\chi_{U_i}$,  
	$$
	f_i=f-\mu(R_i)^{-1}\int h_id\mu \chi_{R_i}\,,\qquad   b^{\lambda}=\sum_if_i\qquad {\rm{and}}\qquad g^{\lambda}=f-b^{\lambda}\,.
	$$
	Arguing as we did in the proof of Proposition \ref{atomo} we can show that  
	$$\|g^{\lambda}\|_{L^{\infty}}\leq C\,\lambda\qquad{\rm{and}}\qquad \frac{1}{\mu(R_i)}\int_{R_i}|h_i|^pd\mu\leq C\,\lambda^p \,  .$$
	If $p_1<\infty$, then
	\begin{align*}
	\|g^{\lambda}\|_{L^{p_1}}^{p_1}&\leq \int_{(\bigcup \tilde R_i)^c}|f|^{p_1}d\mu+\sum_i\int_{\tilde R_i}|f-\sum_jf_j  |^{p_1}d\mu
	= I+II\,.
	 	\end{align*} 
	To estimate $I$ we notice that $(\cup_i\tilde R_i)^c\subset \Omega_{\lambda^p}^c$, so that
	$$
	I\leq  \int_{(\bigcup \tilde R_i)^c}|f|^{p_1-p}|f|^pd\mu\leq C\lambda ^{p_1-p}\|f\|_{L^p}^p\,.
	$$	
	To estimate $II$ we first observe that given $i$ and $x\in\tilde R_i$ there exists 	only two indeces $i_0\leq i$ and $i_1$ such that $x\in U_{i_0}$ and $x\in R_{i_1}$. If $i_0=i_1$, then 
	 $$
	 f(x)-\sum_jf_j(x)=f(x)-f_{i_0}(x)=\mu(R_{i_0})^{-1}\int h_{i_0}d\mu\,,
	 $$
so that $|f(x)-\sum_jf_j(x)|\leq C\lambda$. If $i_0\neq i_1$, then 
$$
f(x)-\sum_jf_j(x)=f(x)-f_{i_0}(x)-f_{i_1}(x)=\mu(R_{i_1})^{-1}\int h_{i_1}d\mu\,,
$$		
so that $|f(x)-\sum_jf_j(x)|\leq C\lambda$. It follows that 
$$
II\leq C\sum_i\int_{\tilde R_i}\lambda ^{p_1}d\mu\leq C\lambda ^{p_1}\sum_i\mu(\tilde R_i)\leq C\lambda^{p_1-p}\|f\|_{L^p}^p\,.
$$
In conclusion, $\|g^{\lambda}\|_{L^{p_1}}^{p_1}\leq C\lambda^{p_1-p}\|f\|_{L^p}^p$\,.

	We now prove that $b^{\lambda}$ is in $H^{1,p}(\mu)$. Indeed, for any $i$, $f_i$ is supported in $\tilde R_i$, has vanishing integral and
	$$
	\|f_i\|_{L^p}\leq C\|h_i\|_{L^p}+C\mu(R_i)^{-1} \mu(\tilde R_i)^{1/p}\int|h_i|d\mu\mu(\tilde R_i)^{1/p}\leq C\lambda \,.
	$$
	This shows that $f_i\in H^{1,p}(\mu)=H^1(\mu)$ and $\|f_i\|_{H^1}\leq C\,\lambda\,\mu(\tilde R_i)$. Since $b^{\lambda}=\sum_if_i$, $b^{\lambda}$ is in $H^{1}(\mu)$ and 
	$$\|b^{\lambda}\|_{H^{1}}\leq C\,\lambda\,\sum_i\mu(\tilde R_i)\leq C\,\lambda\,\frac{\|f\|_{L^p}^p}{\lambda^p}\,,$$
	as required.
	
	We now prove (ii). Fix $t>0$. For any positive $\lambda$, let $f=g^{\lambda}+b^{\lambda}$ be the decomposition of $f$ in $L^{p_1}(\mu)+H^1(\mu)$ given by (i). Thus 
	\begin{align*}
	K(t,f;H^1(\mu),L^{p_1}(\mu))&\leq \inf_{\lambda>0} \big( \|b^{\lambda}\|_{H^1}+t\,\|g^{\lambda}\|_{L^{p_1}} \big)\\
	&\leq C\,\inf_{\lambda>0}\big(\lambda^{1-p}\,\|f\|_{L^p}^{p}+t\,\lambda^{1-p/p_1}\|f\|_{L^p}^{p/p_1}  \big)\\
	%&\leq C\,\|f\|_{L^p}^{p/{p_1}}\,\inf_{\lambda>0}\big(\lambda^{1-p}\,\|f\|_{L^p}^{p(1-1/{p_1})}+t\,\lambda^{1-p/p_1}  \big)\\
	&=C\,\|f\|_{L^p}^{p/{p_1}}\,\inf_{\lambda>0} G(t,\lambda)\,,
	\end{align*}
	where $G(t,\lambda)=\lambda^{1-p}\,\|f\|_{L^p}^{p(1-1/{p_1})}+t\,\lambda^{1-p/p_1}$. Since
	\begin{align*}
	\partial_{\lambda}G(t,\lambda)
	=\lambda^{-p} \big[(1-p)\,\|f\|_{L^p}^{p(1-1/{p_1})}+(1-p/p_1)t\,\lambda^{-p/p_1+p}  \big]\,,
	\end{align*}
	we have that if $p_1<\infty$, then
	$$\inf_{\lambda>0} G(t,\lambda)=G\big(t,C_p\|f\|_{L^p}\,t^{p_1/{p-pp_1}}\big)=C_p\,\|f\|_{L^p}^{1-p/p_1}\,t^{\frac{p_1(p-1)}{p(p_1-1)}}\,.$$
	If $p_1=\infty$, then
	$$\inf_{\lambda>0} G(t,\lambda)=G\big(t,C_p\|f\|_{L^p}\,t^{-1/p}\big)=C_p\,\|f\|_{L^p}\,t^{1-1/p}\,.$$
	It follows that
	$$K(t,f;H^1,L^{p_1})\leq C_p\,\|f\|_{L^p}\,t^{\theta}\,,$$
	proving (ii). This implies that $\|t^{-\theta}\,K(t,f;H^1(\mu),L^{p_1}(\mu))\|_{L^{\infty}}\leq C_p\,\|f\|_{L^p}$, so 
	that $f\in [H^1(\mu),L^{p_1}(\mu)]_{\theta,\infty}$ and $\|f\|_{\theta,\infty}\leq C_p\|f\|_{L^p}$, as required in (iii).
\end{proof}
Following closely the proof of \cite[Theorem ]{V} we deduce from Lemma \ref{intpinfty} the following result.
\begin{theorem}\label{realintH1Lp2}
	Let $1<p<p_1\leq \infty$ and $\theta\in(0,1)$ be such that $\frac{1}{p}=1-\theta+\frac{\theta}{p_1}$. Then 
	$$\big[H^1(\mu),L^{p_1}(\mu)  \big]_{\theta,p}=L^p(\mu)\,.$$
\end{theorem}

 \subsection{Boundedness of singular integrals on $H^1(\mu)$}\label{SubSecSingInt}
In this subsection we prove that integral operators whose kernels satisfy a suitable integral H\"ormander condition are bounded from $H^1(\mu)$ to $L^1(\mu)$. %Note that the integral H\"ormander condition which we require below is weaker than the integral conditions in the hypothesis of \cite[Theorem 1.2]{HS}. 

\begin{theorem}\label{TeolimH1L1}
Let $T$ be a linear operator which is bounded on $L^2(\mu)$ and admits a locally integrable kernel $K$ off the diagonal that satisfies the condition 
\begin{align}\label{stimaH}
\sup_{\tilde R}\sup_{y,\,z\in \tilde R}   \int_{(\tilde R^*)^c}|K(x,y)-K(x,z)| \,d\mu (x) &<\infty\,,
\end{align} 
where the supremum is taken over alla Calder\'on-Zygmund sets $\tilde R$ and $\tilde R^*$ is defined as in Definition \ref{tildeQ*}. Then $T$ extends to a bounded operator from $H^{1}(\mu)$ to $L^1(\mu)$.

\end{theorem}
\begin{proof}
Using (\ref{stimaH}), by \cite[Theorem 1.2]{HS} it is easy to prove that the operator $T$ is of weak type $(1,1)$. Then it is enough to show that there exists a constant $C$ such that $\|Ta\|_{L^1}\leq C$ for any $(1,\infty)$-atom $a$. 

Let $a$ be a $(1,\infty)$-atom supported in the \CZ set $\tilde R$. Recall that $\tilde R\subset {B(x_R, 8h(\tilde R))}$, and $\tilde R^*$ denote the dilated set $\{x\in \mathcal V:~d(x,\tilde R)<h(\tilde R)/4\}$. We estimate the integral $\int_{\mathcal V} |Ta|d\mu$.

We first estimate the integral on $\tilde R^*$ by the Cauchy-Schwarz inequality and the size estimate of the atom:
\begin{align}\label{suR^*}
\int_{R^*} |Ta|d\mu&\leq \|Ta\|_{L^2}\,\mu(\tilde R^*)^{1/2}\nonumber \leq C \, \opnorm T\opnorm_{L^2\rightarrow L^2}\,\|a\|_{L^2}\,\mu(\tilde R)^{1/2}\nonumber\\
%&\leq \kappa_0^{1/2}\,\opnorm T\opnorm_{2}\,\rho(R)^{-1+1/2}\,\rho(R)^{1/2}\nonumber\\
&\leq C\,\opnorm T\opnorm_{L^2\rightarrow L^2}\,.
\end{align}
We consider the integral on the complementary set of $\tilde R^*$ by using the fact that $a$ has vanishing integral: 
$$
\begin{aligned}
\int_{\tilde R^{*c}} |Ta|d\mu&\leq \int_{(\tilde R^{*})^c}\Big|\int_{\tilde R} K(x,y)\,a(y)d\mu(y)  \Big|d\mu(x)\\
&=
\int_{(\tilde R^{*})^c}\Big|\int_{\tilde R} [K(x,y)-K(x,x_R)]\,a(y)d\mu(y)  \Big|d\mu(x)\\
&\leq \int_{(\tilde R^{*})^c}\int_{\tilde R} |K(x,y)-K(x,x_R)|\,|a(y)|d\mu(y)d\mu(x)\\
&=\int_{\tilde R}|a(y)|\Big( \int_{(\tilde R^{*})^c} |K(x,y)-K(x,x_R)|d\mu(x) \Big)d\mu(y)\\
&\leq  \|a\|_{L^1}\,\sup_{y\in \tilde R}\int_{(\tilde R^{*})^c}|K(x,y)-K(x,x_R)|d\mu(x)\\
&\leq  C\,,
\end{aligned}
$$
as required.
\end{proof}
{\bf{Remark}}: The previous result applies to singular integral operators associated with the Laplacian $\mathcal L$ on the tree defined for every function $f:\mathcal V\rightarrow \mathbb C$ by  
\begin{equation}
\mathcal Lf(x)=	f(x)-
 \frac{1}{2\sqrt{q}} \sum_{y \in \mathcal{V} : d(x,y) = 1} q^{\frac{\ell(y) - \ell(x)}{2}} f(y) \qquad \forall x \in \mathcal{V}\,.
	\end{equation}
The Laplacian $\mathcal L$ is bounded on $L^p(\mu)$ for every $p\in [1,\infty]$, it is self-adjoint on $L^2(\mu)$ and its spectrum on $L^2(\mu)$ is $[0,2]$. 
Suppose that $M:\mathbb R\rightarrow \mathbb C$ is bounded supported in $[0,2)$ and satisfies the following Mikhlin-H\"ormander condition of order $s>3/2$ 
	\begin{equation} \label{DefMikhlinHormander}
	\sup_{t>0} \| (D_tM) \phi \|_{W^s_2} < \infty\,,
	\end{equation}
	for some $\phi \in C^\infty_c ([\tfrac{1}{2}, 4])$, $\phi \neq 0$, where $(D_tM)(\lambda)=M(t\lambda)$ and $W^s_2$ denotes the Sobolev space of order $s$ modelled on $L^2([0,2])$. Then the operator $M(\mathcal L)$ extends to a bounded operator from $H^1(\mu)$ to $L^1(\mu)$. Indeed, it was shown in \cite[Theorem 2.3]{HS} that the integral kernel of the operator $H(\mathcal L)$ satisfies condition \eqref{stimaH}. 
For every function $f  :\mathcal V\rightarrow \mathbb C $ we also define the gradient $\triangledown f$ by the formula:
		\begin{equation}
		(\nabla f)(x) = \sum_{y \in \mathcal{V} : d(x,y) = 1} \left| f(y) - f(x) \right| \qquad \forall x \in \mathcal{V}\,.
		\end{equation}
Then the first order Riesz transform $\nabla \mathcal L^{-1/2}$ extends to a bounded operator from $H^1(\mu)$ to $L^1(\mu)$. Indeed, it was shown in \cite[Theorem 2.3]{HS} that the integral kernel of this operator satisfies condition \eqref{stimaH}.

\end{document}